\documentclass[12pt]{amsart}
\usepackage{graphicx}

\usepackage{amsmath,amssymb,amsfonts,amsthm,amsopn}

\newcommand{\field}[1]{\mathbb{#1}}
\newcommand{\tf}{time-frequency}
\newcommand{\stft}{short-time Fourier transform}
\newcommand{\bR}{\field{R}}

\newcommand{\bZ}{\field{Z}}

\newcommand{\aaf}{A_a^{\varphi_1,\varphi_2}}
\newcommand{\modsp}{modulation space}
\newcommand{\aw}{Anti-Wick}
\def\endproof{\hfill\vbox{\hrule height0.6pt\hbox{%
			\vrule height1.3ex width0.6pt\hskip0.8ex
			\vrule width0.6pt}\hrule height0.6pt
	}\medskip}
\newcommand{\tfr}{time-frequency representation}
\def\rd{\bR^d}
\def\rd{\mathbb{R}^d}
\def\f{\varphi}

\def\la{\lambda}

\def\cS{\mathcal{S}}

\def\a{\aleph}

\def\rd{\bR^d}

\def\rdd{{\bR^{2d}}}
\def\zdd{{\bZ^{2d}}}

\def\lrd{L^2(\rd)}
\def\lrdd{L^2(\rdd)}
\def\zd{\bZ^d}

\def\intrd{\int_{\rd}}
\def\intrdd{\int_{\rdd}}

\def\<{\left<}
\def\>{\right>}

\def\mv1{M_v^1}

\def\Mmpq{M_m^{p,q}}
\def\phas{(x,\omega )}

\def\o{\omega}
\def\a{\alpha}
\def\b{\beta}
\def\z{\zeta}

\def\Ren{\mathbb{R}^d}
\def\Renn{\mathbb{R}^{2d}}

\def\sch{\mathcal{S}}

\def\la{\langle}
\def\ra{\rangle}

\def\vs{v_s}

\def\gaw{A^{\f_1,\f_2}_a}
\newcommand{\tfs}{time-frequency shift}
\newtheorem{theorem}{Theorem}[section]
\newtheorem{lemma}[theorem]{Lemma}
\newtheorem{corollary}[theorem]{Corollary}
\newtheorem{proposition}[theorem]{Proposition}

\newtheorem{definition}[theorem]{Definition}

\newtheorem{remark}[theorem]{Remark}

\begin{document}

\title{Note on the Wigner distribution and Localization Operators in the quasi-Banach setting}

\author{Elena Cordero}
\address{Universit\`a di Torino, Dipartimento di
	Matematica, via Carlo Alberto 10, 10123 Torino, Italy} \email{elena.cordero@unito.it}
\thanks{}
%
%


\begin{abstract}Time-frequency analysis have played a crucial role in the development of localization operators in the last twenty years. We present  its applications to the study of boundedness and Schatten Class property for such operators.  In particular,  new sufficient conditions for such operators to belong to the  Schatten-von Neumann Class $S_p(\lrd)$, $0<p<1$, are exhibited.
	As a byproduct, sharp continuity results for the Wigner distribution are also presented.
	\end{abstract}

\maketitle


\section{Introduction}
Localization operators have  a long-standing tradition among physicists, mathematicians and engineers. A special form of such operators called   ``Anti-Wick operators'' had been used as a
quantization procedure  by Berezin~\cite{Berezin71,Shubin91} in 1971.
The terminology ``Time-frequency localization operators'' or simply ``localization operators'' is due to Daubechies, who wrote the popular papers \cite{DB1,DB2} appeared in 1988. From then onwards so many authors have written contributions on this topic that it is not possible to cite them all. In this note we shall focus on the time-frequency properties of such operators and we will  exhibit the results known so far. Much has been done in terms of necessary and sufficient conditions for boundedness of such operators on suitable normed spaces, as well as their belonging to the Schatten-von Neumann Class $ S_p(\lrd)$, $1<p\leq \infty$. Here we focus on the quasi-Banach setting $0<p<1$ and present outcomes in this framework, while reviewing also the known results for the Banach case $p\geq 1$.

 First, we introduce the main features of this study.

The protagonists of time-frequency\, analysis are the operators of
translation and modulation defined by
\begin{equation}
\label{eqi1}
T_xf(t)=f(t-x)\quad{\rm and}\quad M_{\o}f(t)= e^{2\pi i \o
	t}f(t), \, \quad f\in\lrd.
\end{equation}
For a fixed non-zero  $g \in \cS (\rd )$ (the Schwartz class),
the \stft, in short STFT, of $f \in
\cS ' (\rd ) $ (the space of tempered distributions), with respect to the window $g$, is given by
\begin{equation}
\label{eqi2}
V_gf(x,\o)=\la f,M_\o T_x g\ra =\int_{\Ren}
f(t)\, {\overline {g(t-x)}} \, e^{-2\pi i\o t}\,dt\, .
\end{equation}
By means of the STFT, the \tf\ localization operator $\aaf $ with
symbol $a$, analysis  window function $\f _1$, and synthesis
window function $\f _2$ can be formally defined  as
\begin{equation}
\label{eqi4}
\aaf f(t)=\int_{\Renn}a \phas V_{\f \f_1}f \phas M_\omega T_x \f
_2 (t) \, dx d\omega.
\end{equation}
 In particular, if $a \in \cS '(\rdd )$ and $\f _1, \f _2 \in \cS (\rd )$, then
\eqref{eqi4} is a well-defined continuous operator from $\cS
(\rd )$ to $\cS ' (\rd )$.
If $\f _1(t)  = \f _2 (t) = e^{-\pi t^2}$, then $A_a = \aaf $ is the
classical \aw\ operator and the mapping $a \mapsto \aaf  $ is understood
as a quantization rule, cf. \cite{Berezin71,Shubin91} and the recent contribution \cite{deGossonATFA19}.

In a weak  sense, the definition of $\aaf $
in \eqref{eqi4} can be rephrased as
\begin{equation}\label{anti-Wickg}
\la \gaw f,g\ra=\la aV_{\f_1}f, V_{\f_2}g\ra=\la
a,\overline{V_{\f_1}f}\,  V_{\f_2}g\ra,\quad f,g\in\sch(\Ren)\, .
\end{equation}

The definition in \eqref{eqi4} has suggested the study of localization operators as a multilinear
mapping
\begin{equation}\label{map}
(a, \f _1, \f _2) \mapsto \aaf .
\end{equation}
In \cite{CG02,CG05,Wignersharp2018,Nenad2015,Nenad2016,Wong02} the boundedness of the map in \eqref{map} has been
widely studied, in dependence on the function spaces of both
symbol $a$ and windows $\f_1,\f_2$. The sharpest Schatten-class results   are obtained by choosing 
\modsp s as spaces for both symbol and windows, as observed in \cite{CG05} and \cite{Wignersharp2018};  in those contributions the focus is limited to the Banach framework. Sharp compactness results  for localization operators  are contained in \cite{FG2006}. Finally, smoothness and decay of eigenfuctions for localization operators are studied in \cite{BCN19}, see also \cite{Abreu2012,Abreu2016,Abreu2017}.\par   Modulation  spaces are
(quasi-)Banach spaces that measure the concentration of functions and
distributions on the time-frequency plane. Since the STFT is the mean to
extract the time-frequency features of a function/distribution, the idea
that leads to the definition of \modsp s is the following:  \emph{
	give a (quasi)norm} to the STFT. These spaces will be introduced  in the following Subsection $2.2$.\par 
Another way to introduce localization operators is as a form of Weyl transform. The latter can be defined by means of another popular \tfr, the cross-Wigner distribution. Namely, given two functions $f_1,f_2\in
\cS(\Ren)$, the {\it cross-Wigner distribution} $W(f_1,f_1)$  is defined to be
\begin{equation}
\label{eq3232}
W(f_1,f_2)(x,\o)=\int f_1(x+\frac{t}2)\overline{f_2(x-\frac{t}2)} e^{-2\pi
	i\o t}\,dt.
\end{equation}
The quadratic expression $Wf = W(f,f)$ is called the Wigner
distribution of $f$. 

Every continuous operator from $\cS (\rd )$ to $\cS ' (\rd )$ can be
represented as a pseudodifferential operator in the  Weyl form $L_\sigma$  and the connection with the cross-Wigner distribution is provided  by
\begin{equation}\label{equiv1}
\la L_\sigma f,g\ra=\la \sigma,W(g,f)\ra,\quad\quad f,g\in\sch(\Ren).
\end{equation}
Localization operators $\gaw$ can be represented as  Weyl operators as follows (cf. \cite{BCG2004,CG02,Nenad2016})
\begin{equation}\label{WA}\gaw=L_{a\ast W(\f_2,\f_1)},
\end{equation}  so that the Weyl symbol of the localization operator $\aaf$
is given by
\begin{equation}\label{eq2}
\sigma = a\ast W(\f_2,\f_1)\, .
\end{equation}
This representation of localization operators in the Weyl form, together with boundedness properties of Weyl operators  and sharp  continuity properties for the cross-Wigner distribution, yields to  Schatten-class results for localization operators. In particular here we present new outcomes  in the quasi-Banach setting, while reviewing the known results in the Banach framework, see Theorems \ref{class} and \ref{main} below. \par
The paper is organized as follows.   Section $2$  presents the basic definitions and properties of the Schatten-von Neumann Classes $S_p(\lrd)$, $0<p\leq\infty$, of the modulation spaces and  the time-frequency analysis tools needed to infer  our results.  
Section $3$ exhibits the sufficient conditions for localization operators to be in the Schatten-von Neumann classes $S_p$. To chase this goal, sharp continuity  properties for the cross-Wigner distribution are presented. Such result is new  in the framework of quasi-Banach modulation spaces and is the main ingredient  to prove sufficient Schatten class conditions for localization operators. Section $4$ contains necessary Schatten class results for localization operators and ends by showing  perspectives and open problems about this topic.
\section{Preliminaries on Schatten Classes, Modulation Spaces and Frames}
\subsection{Schatten-von Neumann Classes.}
We limit to consider the Hilbert space $\lrd$. Let $T$ be a compact operator on $\lrd$. Then $T^\ast T$:  $\lrd\to   \lrd$  is compact, self-adjoint,
and non-negative. Hence, we can define the absolute value of $T$ by $|T| = (T^\ast T)^{\frac12}$, acting on $\lrd$. Recall that $|T|$ is compact, self-adjoint, and non-negative, hence by the Spectral Theorem we can find  an orthonormal basis $(\psi_n)_n$ for $\lrd$  consisting of
eigenvectors of $|T|$. The corresponding eigenvalues $s_1(T) \geq  s_2(T)\geq \dots \geq  s_n(T) \geq \dots\geq 0$, are called the  the singular values of $T$. 

If $0<p<\infty$ and the sequence of singular values is $\ell^p$-summable, then $T$ is said to belong to the Schatten-von Neumann class $S_p(\lrd)$. If $1 \leq p < \infty$, a norm is
associated to $S_p(\lrd)$ by 
\begin{equation}\label{normSp}
\|T\|_{S_p}:=\left(\sum_{n=1}^\infty s_n(T)^p\right)^\frac1p.
\end{equation}
If $1\leq p<\infty$ then $(S_p(\lrd), \|\cdot\|_{S_p})$ is a Banach space whereas, for $0<p<1$, $(S_p(\lrd), \|\cdot\|_{S_p})$ is a quasi-Banach space since  the quantity $\|T\|_{S_p}$ defined in \eqref{normSp} is only a quasinorm. 

For completeness, we define  $S_\infty(\lrd)$ to be the space of bounded operators on $\lrd$. The Schatten-von Neumann classes are nested, with
$S_p \subset S_q$, for details on this topic we refer to \cite{Gohberg1969,Reed-Simon1975,Simon2005,Schatten70,Shubin91,zhu2007operator}.

For $2\leq p<\infty$ and   $T$ in  $S_p(\lrd)$, we can express  its
norm by
\begin{equation}\label{char-p-big2}
\|T\|_{S_p}^p=\sup \sum_n \|T \phi_n\|_{L^2}^p,
\end{equation}
the supremum being over all orthonormal bases $(\phi_n)_n$ of
$\lrd$. Then, it is a straightforward consequence (see \cite[Theorem 12]{ZhuSchatten2015})
\begin{equation}\label{Schatten-norm-bound}
\left(\sum_{n} |\la T
\phi_n, \phi_n\ra|^p\right)^{1/p}\leq \|T\|_{S_p},
\end{equation} for every
orthonormal basis $(\phi_n)_n$, $2\leq p<\infty$. If $T\in S_2(\lrd)$ then $T$ is called \emph{Hilbert-Schmidt}  operator.  If $T\in  S_1(\lrd)$ then $T$ is said to be a \emph{trace class} operator and the space  $S_1$ is named the Trace Class. \par 
\begin{remark}\label{counterexp}  For $0<p<2$, the  characterization in \eqref{char-p-big2} does not hold, in general.  In fact, a simple example is shown by Bingyang, Khoi and Zhu in the paper \cite{ZhuSchatten2015}. Let us recall it for sake of clarity in the case of the Hilbert space $H=\lrd$. Fix an orthonormal basis $(\phi_n)_n$
and consider the function $h\in \lrd$ given by
$$h=\sum_{n=1}^\infty \frac{\phi_n}{\sqrt{n}\log( n+1)}. $$
Define the rank-one operator on $\lrd$ by 
$$ Tf= \la f,h\ra h, \quad f\in\lrd.$$
We have
$$T \phi_n=  \la \phi_n ,h\ra h= \frac{h}{\sqrt{n}\log( n+1)},\quad n\geq 1.$$
It follows that
$$\sum_{n=1}^\infty \|T\phi_n\|_{L^2}^p=\|h\|_{L^2}^p\sum_{n=1}^\infty \frac{1}{[\sqrt{n}\log(n+1)]^p}=\infty$$
for any $0<p<2$.
\end{remark} 


\subsection{Modulation Spaces}
\subsubsection{Weight functions}
In the sequel $v$ will always be a
continuous, positive,  submultiplicative  weight function on $\rd$, i.e., 
$ v(z_1+z_2)\leq v(z_1)v(z_2)$, for all $ z_1,z_2\in\Ren$.
We say that $m\in \mathcal{M}_v(\rd)$  if $m$ is a positive, continuous  weight function  on $\Ren$ {\it
	$v$-moderate}:
$ m(z_1+z_2)\leq Cv(z_1)m(z_2)$  for all $z_1,z_2\in\Ren$.
We will mainly work with polynomial weights of the type
\begin{equation}\label{vs}
v_s(z)=\la z\ra^s =(1+|z|^2)^{s/2},\quad s\in\bR,\quad z\in\rd.
\end{equation}
Observe that,  for $s<0$, $v_s$ is $v_{|s|}$-moderate.\par 
Given two weight functions $m_1,m_2$ on $\rd$, we write $$(m_1\otimes m_2)(x,\o)=m_1(x)m_2(\o),\quad x,\o\in \rd.$$
\noindent
{\bf Modulation Spaces.} We present the more general definition of such spaces, containing the quasi-Banach setting,  introduced first by Y.V. Galperin and S. Samarah in \cite{Galperin2004}.
\begin{definition}\label{def2.4}
	Fix a non-zero window $g\in\cS(\rd)$, a weight $m\in\mathcal{M}_v(\rdd)$ and $0<p,q\leq \infty$. The modulation space $M^{p,q}_m(\rd)$ consists of all tempered distributions $f\in\cS'(\rd)$ such that the (quasi)norm 
	\begin{equation}
	\|f\|_{M^{p,q}_m}=\|V_gf\|_{L^{p,q}_m}=\left(\intrd\left(\intrd |V_g f \phas|^p m\phas^p dx  \right)^{\frac qp}d\o\right)^\frac1q 
	\end{equation}
	(obvious changes with $p=\infty$ or $q=\infty)$ is finite. 
\end{definition}

The most known modulation spaces  are those  $M^{p,q}_m(\rd)$, with $1\leq p,q\leq \infty$, introduced by H. Feichtinger in \cite{feichtinger-modulation}. In that paper their main properties were exhibited; in particular we recall that  they are Banach spaces, whose norm does not depend on the window $g$:  different window functions in $\cS(\rd)$ yield equivalent norms. Moreover, the window class $\cS(\rd)$ can be extended  to the modulation space  $M^{1,1}_v(\rd)$ (so-called Feichtinger algebra). 

For shortness, we write $M^p_m(\rd)$ in place of $M^{p,p}_m(\rd)$ and $M^{p,q}(\rd)$ if $m\equiv 1$.

The modulation spaces $M^{p,q}_m(\rd)$, $0<p,q<1$, where  introduced almost twenty years later by Y.V. Galperin and S. Samarah in \cite{Galperin2004}.
In this framework, it appears that the largest natural class of windows universally admissible for all spaces $M^{p,q}_m(\rd)$, $0<p,q\leq \infty$ (with weight $m$ having at most polynomial growth) is the Schwartz class $\cS(\rd)$.
Many properties related to the quasi-Banach setting are still unexplored. 

The focus of this paper is on the quasi Banach setting, which allows to infer new results for localization operators.\par 
In the sequel we shall use inclusion relations for modulation spaces (cf. \cite[Theorem 3.4]{Galperin2004} and \cite[Theorem 12.2.2]{grochenig}):
\begin{theorem}\label{inclusionG}
	Let $m\in\mathcal{M}_v(\rdd)$. If $0<p_1\leq p_2\leq \infty$ and $0<q_1\leq q_2\leq \infty$ then $M^{p_1,q_1}_m(\rd)\subseteq M^{p_2,q_2}_m(\rd)$.
\end{theorem}


\begin{remark}
	In our framework it is important to notice the following inclusion relation for $s>0$:
\begin{equation}\label{compact}
	 M^{\infty}_{v_s\otimes 1} (\rdd) \subset M^{p,\infty}(\rdd)\quad  \mbox{if}\, \,p>2d/s.
\end{equation}
	This follows from the recent contribution \cite[Theorem 1.5]{Guo2019}. 
\end{remark}

Let us recall convolution relations for modulations spaces. They are contained  in  the contributions \cite{CG02} and \cite{toft2004} for the Banach framework. The more general case is exhibited in \cite{BCN19}.
\begin{proposition}\label{mconvmp}
	Let $\nu (\omega )>0$ be  an arbitrary  weight function on $\Ren$,  $0<
	p,q,r,t,u,\gamma\leq\infty$, with
	\begin{equation}\label{Holderindices}
	\frac 1u+\frac 1t=\frac 1\gamma,
	\end{equation}
	and 
	\begin{equation}\label{Youngindicesrbig1}
	\frac1p+\frac1q=1+\frac1r,\quad \,\, \text{ for } \, 1\leq r\leq \infty
	\end{equation}
	whereas
	\begin{equation}\label{Youngindicesrbig1}
	p=q=r,\quad \,\, \text{ for } \, 0<r<1.
	\end{equation}
	For $m\in\mathcal{M}_v(\rdd)$,   $m_1(x)
	= m(x,0) $ and $m_2(\omega ) = m(0,\omega )$ are the restrictions
	to $\Ren\times\{0\}$ and  $\{0\}\times\Ren$, and likewise for $v$. Then 
	\begin{equation}\label{mconvm}
	M^{p,u}_{m_1\otimes \nu}(\Ren)\ast  M^{q,t}_{v_1\otimes
		v_2\nu^{-1}}(\Ren)\subseteq M^{r,\gamma}_m(\Ren)
	\end{equation}
	with  norm inequality  $$\| f\ast h \|_{M^{r,\gamma}_m}\lesssim
	\|f\|_{M^{p,u}_{m_1\otimes \nu}}\|h\|_{ M^{q,t}_{v_1\otimes
			v_2\nu^{-1}}}.$$
\end{proposition}

\subsection{Frame Theory.} A sequence of functions $\{b_j\,:\,j\in\mathcal{J}\} $ in $\lrd$ is a \emph{frame} for the Hilbert space $\lrd$ if there exist
	positive constants $0<A\leq B<\infty$, such that
	\begin{equation}\label{frame}
	A\|f\|_{L^2}^2\leq\sum_{j\in\mathcal{J}}|\la f, b_j\ra|^2\leq B
	\|f\|_{L^2}^2,\quad \forall f\in\lrd.
	\end{equation}
	The constants $A$ and $B$ are called \emph{lower} and \emph{upper}
	frame bounds, respectively. It is straightforward from
	\eqref{frame} (or see, e.g., \cite[Pag. 398]{HeWeiss96})  to check the
	elements of a frame satisfy
	\begin{equation}\label{unif}
	\|b_j\|_{L^2}\leq \sqrt{B},\quad \forall j \in\mathcal{J}.
	\end{equation}
Using \eqref{unif},   in \cite{CG05}  we extended the inequality in \eqref{Schatten-norm-bound}  from orthonormal bases to frames.
\begin{lemma}\label{ppp}
	Let $(b_n)_{n}$ be a frame for $\lrd$, as defined
	in \eqref{frame}, with upper bound $B$. If $T\in S_p(\lrd)$, for $1\leq p\leq\infty$, then
	\begin{equation}\label{adj}
	\left(\sum_{n=1}^\infty |\la T b_n,b_n\ra|^p\right)^{1/p}\leq
	B\|T\|_{S_p}.
	\end{equation}
\end{lemma}

Observe that an orthonormal basis is a special instance of frame with upper bound $B=1$; hence Lemma \ref{ppp} provides an alternative proof to the inequality in  \eqref{Schatten-norm-bound}, for every $1\leq p\leq\infty$.

In  the case $0<p<1$, Lemma \ref{ppp} is false in general. This is a straightforward consequence of the following result \cite[Proposition 22]{ZhuSchatten2015}:
\begin{proposition} Suppose $0 < p < 1$ and $(\phi_n)_n$ any orthonormal basis for $\lrd$.
Then there exists a positive operator $S \in S_p(\lrd)$ such that $(\la S\phi_n,\phi_n\ra )_n\notin \ell^p$.
\end{proposition}
Since an orthonormal basis is a frame with frame bounds $A=B=1$, it follows that the majorization \eqref{adj} fails for  $(\phi_n)_n$ and, consequently, Lemma \ref{ppp} is false.
For $p\geq 1$, a useful consequence of Lemma \ref{ppp} is as follows (cf.  \cite[Corollary 2]{CG05}):
\begin{corollary}\label{adj4}
	Let $(b_n)_n$ be a frame with upper bound $B$. Let $L\in S_\infty(\lrd)$
	and $T\in  S_p(\lrd)$, with $1\leq p\leq\infty$. Then we have
	\begin{equation}\label{adj6}
	\left(\sum_{n=1} ^\infty|\la T b_n,L
	b_n\ra|^p\right)^{1/p}\leq  B\|T\|_{S_p}\|L\|_{S_\infty}.
	\end{equation}
\end{corollary}

 In  \cite[Proposition 10]{Maurice2015}, see also \cite{seip92,seip-wallsten}, it is
proved that, if $\a\b<1$ and
\begin{equation}\label{fi} \f:=2^{d/4}e^{-\pi
	x^2},\end{equation}
 then the set
of the Gaussian \tfs \,$(M_{\b n}T_{\a k}\f)_{n,k\in \bZ^d}$ is a
frame for $L^2(\rdd)$ (called Gabor frame). 
In the sequel we shall also use the  Gabor frames on  $\lrdd$ given by
$$ (M_{\b n}T_{\a k}\Phi)_{k,n
	\in\zdd},$$ where $\Phi$ is the $2d$-dimensional
Gaussian function below
\begin{equation}\label{Phi}
\Phi\phas:=2^{-d}e^{-\pi (x^2+\o^2)},\quad \phas\in\rdd.
\end{equation}
It is easy to compute (or see, e.g.,
\cite[Lemma 1.5.2]{grochenig}) that
\begin{equation}\label{Gauss0}V_\f \f\phas=2^{-d/2}
e^{-\pi i x \o}e^{-\frac\pi 2(x^2+\o^2)}.\end{equation}
\begin{definition}\label{ellp}
	For $0<p,q\leq \infty$, $m\in \mathcal{M}_v(\zdd)$, the space $\ell^{p,q}_m(\zdd)$ consists of all sequences $c=(c_{k,n})_{k,n\in\zd}$ for which the (quasi-)norm 
	$$\|c\|_{\ell^{p,q}_m}=\left(\sum_{n\in\zd}\left(\sum_{k\in\zd}|c_{k,n}|^p m(k,n)^p\right)^{\frac qp}\right)^{\frac 1q}
	$$
	(with obvious modification for $p=\infty$ or $q=\infty$) is finite.
\end{definition}

For $p=q$, $\ell^{p,q}_m(\zdd)=\ell^p_m(\zdd)$, the standard spaces of sequences.  Namely, in dimension $d$, 
for $0<p\leq \infty$, $m$ a weight function on $\zd$, a sequence $c=(c_k)_{k\in\zd}$ is in $\ell^p_m(\zd)$ if
$$\|c\|_{\ell^{p}_m}=\left(\sum_{k\in\zd}|c_{k}|^p m(k)^p\right)^{\frac 1p}<\infty.
$$

Discrete equivalent modulation spaces norms are produced by means of Gabor frames.  The key result is the following characterization for the $M^{p,q}_m$- norm of localization symbols (see \cite[Chapter 12]{grochenig} for $1\leq p,q\leq\infty$, and \cite[Theorem 3.7]{Galperin2004} for $0<p,q<1$). 
\begin{theorem}\label{framesmod}
	Assume $m\in\mathcal{M}_v(\rdd)$,  $0<p,q\leq\infty$. Consider the Gabor frame  $(M_{\b n}T_{\a k}\Phi)_{k,n
		\in\zdd}$ with Gaussian  window $\Phi$ in \eqref{Phi}. Then, for every $a \in\Mmpq(\rdd)$,
	\begin{equation}\label{idea}
	\|a\|_{M^{p,q}_m(\rdd)}\asymp \|(\la a ,M_{\b n}T_{\a
		k}\Phi\ra_{n,k\in\zdd})_{n,k\in\zdd}\|_{\ell_m^{p,q}(\bZ^{4d})}.
	\end{equation}
\end{theorem}

\subsection{Time-frequency Tools}
In the sequel we shall need to compute the STFT of the cross-Wigner distribution, contained below \cite[Lemma 14.5.1]{grochenig}: 
\begin{lemma}\label{STFTSTFT}
	Fix a nonzero  $g \in \cS (\Ren ) $ and let  $\Phi=W (g , g ) \in\sch(\Renn)$. Then the STFT of $W(f _1,
	f _2) $ with respect to the window $\Phi $ is given by
	\begin{equation}
	\label{eql4}
	{ {V}}_\Phi (W(f_1,f_2)) (z, \zeta ) =e^{-2\pi i
		z_2\z_2}{\overline{V_{g
			}f_2(z_1+\frac{\z_2}2,z_2-\frac{\z_1}2})}V_{g
	}f_1(z_1-\frac{\z_2}2,z_2+\frac{\z_1}2)\, .
	\end{equation}	
\end{lemma}

 The following properties of the STFT (cf. \cite[Lemma 1]{CG05}) will be used to prove necessary Schatten class conditions for localization operators.
\begin{lemma} 
	If
	$z=(z_1,z_2)\in\Renn$,  $\zeta=(\zeta_1,\zeta_2)\in\Renn$, then
	\begin{align}\label{eqr4}
	T_{(z_1,z_2)}({\overline{V_{\f_1}f}}\cdot V_{\f_2}g ) \phas &=
	{\overline{V_{\f_1}(M_{z_2}T_{z_1}f)}\phas} \,
	V_{\f_2}(M_{z_2}T_{z_1}g)\phas, \,\\
	\label{eqr5}
	M_{(\z_1,\z_2)}\big({\overline{V_{\f_1}f}}\,V_{\f_2}g \big) \phas
	&= \overline{V_{\f_1}f\phas
	}  \, V_{(M_{\z_1}T_{-\z_2}\f_2)}( M_{\z_1}T_{-\z_2}g)\phas, \, \\
	\label{bo}
	M_\zeta T_z({\overline{V_{\f_1}f}}V_{\f_2}g) &=
	{\overline{V_{\f_1}(M_{z_1}T_{z_2}f)}}V_{(M_{\zeta_1}T_{-\zeta_2}\f_2)}(M_{\zeta_1}T_{-\zeta_2}M_{z_1}T_{z_2}g).
	\end{align}
\end{lemma}

\section{Sufficient Conditions for  Schatten Class $S_p$, $0<p\leq \infty$}
In this Section we present sufficient conditions for  Schatten Class properties of localization operators. The Banach case $p\geq 1$ was studied in \cite{CG02,CG05}. The main result (cf. Theorem \ref{class} below) will take care of  the full range $0<p\leq\infty$. 

First, we need to recall similar properties for Weyl operators, obtained in several papers, we refer the interested reader to \cite{CG02,grochenig,GH99,Sjo94,toft2004}.

\begin{theorem}\label{Charly1} For $0<p\leq\infty$, we have:\\
	(i) If $0<  p \leq 2$ and  $\sigma \in M^{p} (\Renn )$, then
	$L_\sigma \in S_p$ and
	$\|L_\sigma \|_{S_p} \lesssim \|\sigma \|_{M^{p}}$.\\
	(ii) If $2\leq p \leq \infty$ and  $\sigma \in M^{p,p'} (\Renn )$, then
	$L_\sigma \in S_p$ and
	$\|L_\sigma \|_{S_p} \lesssim \|\sigma \|_{M^{p,p'}}$.
\end{theorem}
\begin{proof} The proof for $p\geq 1$ can be found in \cite[Theorem 3.1]{CG02}, see also references therein. The case $0<p<1$ is contained in \cite[Theorem 3.4]{ToftquasiBanach2017}.
\end{proof}

We now focus on the properties of the cross-Wigner distribution,  which enjoys the following property. 

\begin{theorem} \label{T1} Assume $p_i,q_i,p,q\in (0,\infty]$, $i=1,2$, $s\in \bR$, such that
	\begin{equation}\label{WIR}
	p_i,q_i\leq q,  \ \quad i=1,2
	\end{equation}
	and that
	\begin{equation}\label{Wigindexsharp}
	\frac1{p_1}+\frac1{p_2}\geq \frac1{p}+\frac1{q},\quad \frac1{q_1}+\frac1{q_2} \geq \frac1{p}+\frac1{q}.
	\end{equation}
	Then,
	if $f_1\in M^{p_1,q_1}_{v_{|s|}}(\Ren)$ and
	$f_2\in M^{p_2,q_2}_{v_s}(\Ren)$ we have  $W(f_1,f_2)\in
	M^{p,q}_{1\otimes v_s}(\Renn)$, and
	\begin{equation}\label{wigest}
	\| W(f_1,f_2)\|_{M^{p,q}_{1\otimes v_s}}\lesssim
	\|f_1\|_{M^{p_1,q_1}_{v_{|s|}}}\| f_2\|_{M^{p_2,q_2}_{v_s}}.
	\end{equation}
	
	\par
	Vice versa, assume that there exists a constant $C>0$ such that
	\begin{equation}\label{Wigestsharp}
	\|W(f_1,f_2)\|_{M^{p,q}}\leq C \|f_1\|_{M^{p_1,q_1}} \|f_2\|_{M^{p_2,q_2}},\quad \forall f_1,f_2\in\cS(\rdd).
	\end{equation}
	Then \eqref{WIR} and \eqref{Wigindexsharp} must hold.
\end{theorem}
\begin{proof} 
	\emph{Sufficient Conditions.} The result  for the indices $p_i,q_i,p,q\in [1,\infty]$ is proved in \cite[Theorem 3.1]{Wignersharp2018}.  The general case follows easily from that one, since the main tool is provided by the inclusion relations for modulation spaces in \eqref{inclusionG}. We detail its steps for sake of clarity.\par  
	First, study the case both $0<p,q<\infty$. Let $g\in \cS (\rd ) $ and set $\Phi=W(g,g)\in\sch(\Renn)$.  If $\zeta
	= (\z_1,\z_2)\in
	\Renn$, we write $\tilde{\zeta } = (\zeta _2,-\zeta _1)$. Then, from 
	Lemma \ref{STFTSTFT}, 
	\begin{equation}\label{e0}
	|{{V}}_\Phi (W(f_1,f_2))(z,\zeta)| =| V_g f_2(z
	+\tfrac{\tilde{\z }}{2})| \,  |V_g f_1(z - \tfrac{\tilde{\z }}{2})| \,.
	\end{equation}
	Hence,
	\begin{equation*}
	\|W( f_1,f_2)\|_{M^{p,q}_{1\otimes v_s}}  \asymp
	\left(\intrdd\!\left(\intrdd\!
	| V_g f_2(z +\tfrac{\tilde{\z }}{2})|^p \,  |V_g f_1(z -
	\tfrac{\tilde{\z }}{2})|^p   \, dz \right)^\frac{q}{p} \, \langle \zeta \rangle
	^{sq} \, d\zeta \right)^{1/q}.
	\end{equation*}
	Making the change of variables $z \mapsto z-\tilde{\zeta } /2$, the
	integral over $z$ becomes the convolution $(|V_g f_2|^p\ast
	|(V_g{f_1})^*|^p)(\tilde{\zeta })$,
	and observing that $(1\otimes v_s) (z,\zeta ) = \langle \zeta \rangle ^s =
	v_s (\zeta )= v_s (\tilde{\zeta })$, we obtain
	\begin{eqnarray*}
		\|W(f_1,f_2)\|_{M^{p,q}_{1\otimes v_s}}
		&\asymp&
		\left(\iint_{\Renn}\!(|V_g f_2|^p\ast |(V_g {f_1})^*|^p)^\frac{q}{p}(\tilde{\zeta})
		v_s(\tilde{\zeta})^{q} \, d\zeta \right)^{1/p}\\
		&=& \| \, |V_g f_2|^p\ast |(V_g{ f_1})^*|^p \, \|^{\frac 1 p}_{L^\frac{q}{p}_{v _{ps}}}.
	\end{eqnarray*}
	Hence 
	\begin{equation}\label{e1}
	\|W(f_1,f_2)\|^p_{M^{p,q}_{1\otimes v_s}}\asymp \| \, |V_g f_2|^p\ast |(V_g{ f_1})^*|^p \, \|_{L^\frac{q}{p}_{v _{ps}}}.
	\end{equation}
	\emph{Case $0<p\leq q<\infty$}.\par \emph{ Step 1.}  Consider first the case $p\leq p_i,q_i$, $i=1,2$, satisfying  the condition
	 \begin{equation}\label{Wigindex}
	\frac1{p_1}+\frac1{p_2}=\frac1{q_1}+\frac1{q_2}=\frac1{p}+\frac1{q},
	\end{equation}
	(and hence $p_i,q_i\leq q$,  $i=1,2$).
	Since $q/p\geq 1$, we can apply Young's Inequality  for mixed-normed spaces \cite{Galperin2014} and majorize \eqref{e1} as follows
	\begin{align*}
	\|W(f_1,f_2)\|^p_{M^{p,q}_{1\otimes v_s}}&\lesssim 
	\| \, |V_g f_2|^p\|_{L^{r_2,s_2}_{v _{p|s|}}}\|\, |(V_g{ f_1})^*|^p  \|_{L^{r_1,s_1}_{v _{ps}}} \, \\
	&= \| |V_g{ f_1}|^p  \|_{L^{r_1,s_1}_{v _{p|s|}}} \| \, |V_g f_2|^p\|_{L^{r_2,s_2}_{v _{ps}}}\, \\
	&= \| V_g{ f_1}  \|^p_{L^{pr_1,ps_1}_{v _{|s|}}} \| V_g f_2\|^p_{L^{p r_2,ps_2}_{v _{s}}}\, ,
	\end{align*}
	for every $1\leq r_1,r_2,s_1,s_2\leq\infty$ such that 
	\begin{equation}\label{e2}
	\frac 1{r_1}+\frac{1}{r_2}=\frac 1{s_1}+\frac{1}{s_2}=1+ \frac{p}{q}.
	\end{equation}
	Choosing $r_i=p_i/p\geq 1$, $s_i=q_i/p\geq 1$, $i=1,2$, the indices' relation \eqref{e2} becomes \eqref{Wigindex}  and we obtain 
	$$	\|W(f_1,f_2)\|_{M^{p,q}_{1\otimes v_s}}\lesssim \| V_g{ f_1}  \|_{L^{p_1,q_1}_{v _{|s|}}} \| V_g f_2\|_{L^{p_2,q_2}_{v _{s}}}\asymp \|f_1\|_{M^{p_1,q_1}_{v_{|s|}}}\|f_2\|_{M^{p_2,q_2}_{v_s}}.
	$$
	Now, still assume $p\leq p_i,q_i$, $i=1,2$ but
	$$\frac 1{p_1}+\frac 1{p_2}\geq \frac{1}{p}+\frac 1q,\quad \frac 1{q_1}+\frac 1{q_2}= \frac{1}{p}+\frac 1q,
	$$
	(hence $p_i,q_i\leq q$, $i=1,2$). We set $u_1=t p_1$,  and look for $t\geq 1$ (hence $u_1\geq p_1$) such that
	$$\frac 1 {u_1}+\frac 1{p_2}= \frac{1}{p}+\frac 1q
	$$
	that gives
	$$0<\frac 1t=\frac{p_1}{p}+\frac{p_1}{q}-\frac{p_1}{p_2}\leq1
	$$
	because $p_1(1/p+1/q)-p_1/p_2\leq p_1(1/p_1+1/p_2)-p_1/p_2=1$ whereas the lower bound of the previous estimate follows by   $1/(tp_1)=1/p+1/q-1/p_2>0$ since $p\leq p_2$.
	Hence the previous part of the proof gives 
	\begin{equation*}
	\|W(f_1,f_2)\|_{M^{p,q}_{1\otimes v_s}}\lesssim \|f_1\|_{M^{u_1,q_1}_{v_{|s|}}}\|f_2\|_{M^{p_2,q_2}_{v_s}}\lesssim \|f_1\|_{M^{p_1,q_1}_{v_{|s|}}}\|f_2\|_{M^{p_2,q_2}_{v_s}},
	\end{equation*}
	where the last inequality follows by inclusion relations for modulations spaces $M^{p_1,q_1}_{v_s}(\rd)\subseteq M^{u_1,q_1}_{v_s}(\rd)$ for $p_1\leq u_1$.
	
	The general case 
	$$\frac 1{p_1}+\frac 1{p_2}\geq  \frac{1}{p}+\frac 1q,\quad \frac 1{q_1}+\frac 1{q_2}\geq \frac{1}{p}+\frac 1q,
	$$
	is similar.\par
	\noindent \emph{ Step 2}. Assume now that $0<p_i,q_i\leq q$, $i=1,2$, and satisfy relation \eqref{Wigindexsharp}.  If at least one out of the indices $p_1, p_2$ is less than $p$, assume for instance $p_1\leq p$, whereas $p\leq q_1,q_2$, then we proceed as follows. 
	We choose $u_1=p$, $u_2=q$, and deduce by the results in Step 1 (with $p_1=u_1$ and $p_2=u_2$) that
	$$ 	\|W(f_1,f_2)\|_{M^{p,q}_{1\otimes v_s}}\lesssim \|f_1\|_{M^{u_1,q_1}_{v_{|s|}}}\|f_2\|_{M^{u_2,q_2}_{v_s}}\lesssim\|f_1\|_{M^{p_1,q_1}_{v_{|s|}}}\|f_2\|_{M^{p_2,q_2}_{v_s}}
	$$ 
	where the last inequality follows by inclusion relations for modulation spaces, since 
	$p_1\leq u_1=p$ and $p_2\leq u_2=q$.\par
	Similarly we argue when at least one out of the indices $q_1, q_2$ is less than $p$ and $p\leq p_1,p_2$ or when at least one out of the indices $q_1, q_2$ is less than $p$ and at least one out of the indices $p_1, p_2$ is less than $p$. The remaining case $p\leq p_i,q_i\leq q$ is treated in Step 1.
	
	\noindent\emph{Case $0<p<q=\infty$.} The argument are similar to the case $0<p\leq q<\infty$.
	
	\noindent\emph{Case $p=q=\infty$.}  We use \eqref{e0} and the submultiplicative property of the weight $v_s$,
	\begin{align*}
	\|W(f_1,f_2)\|_{M^{\infty}_{1\otimes v_s}}&=\sup_{z,\zeta\in\rdd}| V_g f_2(z
	+\tfrac{\tilde{\z }}{2})| \,  |V_g f_1(z - \tfrac{\tilde{\z }}{2})| v_s(\zeta)\\
	&=\sup_{z,\zeta\in\rdd}|| V_g f_2(z)| \,  |(V_g f_1)^*(z - \tilde{\z })| v_s(\z )\\
	&=\sup_{z,\zeta\in\rdd}|| V_g f_2(z)| \,  |(V_g f_1)^*(z - \tilde{\z })| v_s(\tilde{\z })\\
	&\leq \sup_{z\in\rdd}(\|V_g f_1 v_{|s|}\|_{\infty} \,|V_g f_2(z) v_s(z)|)= \|V_g f_1 v_{|s|}\|_{\infty}\|V_g f_2 v_s\|_{\infty}\\
	&\asymp\|f\|_{M^\infty_{v_{|s|}}}\|g\|_{M^\infty_{v_s}}\leq \|f\|_{M^{p_1,q_1}_{v_{|s|}}}\|f\|_{M^{p_2,q_2}_{v_s}},
	\end{align*}
	for every $0< p_i,q_i\leq \infty$, $i=1,2$. 
	
	\noindent\emph{Case $p>q$.} Using the inclusion relations for modulation spaces, we majorize 
	$$\|W(f_1,f_2)\|_{M^{p,q}_{1\otimes v_s}}\lesssim \|W(f_1,f_2)\|_{M^{q,q}_{1\otimes v_s}}\lesssim \|f_1\|_{M^{p_1,q_1}_{v_{|s|}}}\|f_2\|_{M^{p_2,q_2}_{v_s}}$$
	for every $0< p_i,q_i\leq q$, $i=1,2$. Here we have applied  the case $p\leq q$ with $p=q$. Notice that in this case condition \eqref{Wigestsharp} is trivially satisfied, since from $p_1,q_i\leq q$ we infer $1/p_1+1/p_2\geq 1/q+1/q$, $1/q_1+1/q_2\geq 1/q+1/q$. This ends the proof of the sufficient conditions.\par 
	\emph{Necessary Conditions.} The proof works exactly the same as that of \cite[Theorem 3.5]{Wignersharp2018}. In fact, the main point is the use of   the $M^{r,s}$-norm of the rescaled Gaussian  $\f_\lambda(x)=\f(\sqrt{\lambda} x)$, with $\f(x)=e^{-\pi x^2}$, for which we reckon (see  also \cite[Lemma 3.2]{cordero2} and \cite[Lemma 1.8]{toft2004}):
	$$ \|\f_\lambda\|_{M^{r,s}}\asymp \lambda^{-\frac d {2r}}(\lambda+1)^{-\frac d2(1-\frac1s-\frac1r)},$$ 
	for every $0<r,s\leq\infty$.
\end{proof}

Based on the tools developed above, we establish the following
Schatten class  results for localization operators.

\begin{theorem}\label{class} For $s\geq 0$, we have the following statements.\\
	(i) If $0< p  <1$, then  the mapping $(a,\f _1, \f _2) \mapsto
	\aaf $ is  bounded  from
	$M^{p,\infty }_{1\otimes v_{-s}}(\rdd ) \times M^p_{v_s} (\rd )\times
	M^p_{v_s} (\rd )$ into $S_p$:
	$$\|\gaw\|_{S_p}\lesssim
	\|a\|_{M^{p,\infty}_{1\otimes v_{-s}}}\|\f_1\|_{M^p_{\vs}}\|\f_2\|_{M^p_{\vs}}\,
	.$$
	(ii) If $1\leq   p  \leq 2$, then  the mapping $(a,\f _1, \f _2) \mapsto
	\aaf $ is  bounded  from
	$M^{p,\infty }_{1\otimes v_{-s}}(\rdd ) \times M^1_{v_s} (\rd )\times
	M^p_{v_s} (\rd )$ into $S_p$:
	$$\|\gaw\|_{S_p}\lesssim
	\|a\|_{M^{p,\infty}_{1\otimes v_{-s}}}\|\f_1\|_{M^1_{\vs}}\|\f_2\|_{M^p_{\vs}}\,
	.$$
	(iii)  If $2 \leq p  \leq \infty$, then the mapping $(a,\f _1, \f _2) \mapsto
	\aaf $ is  bounded  from
	$M^{p,\infty }_{1\otimes v_{-s}} \times M^1_{v_s}\times M^{p'}_{v_s}$ into
	$S_p$:
	$$\|\gaw\|_{S_p}\lesssim
	\|a\|_{M^{p,\infty}_{1\otimes v_{-s}}}\|\f_1\|_{M^1_{\vs}}\|\f_2\|_{M^{p'}_{\vs}}\,
	.$$
\end{theorem}

\begin{proof}
	({\it i}) If $\f_1 \in  M^{p}_{v_s}(\Ren)$ and $\f_2 \in
	M^{p}_{v_s}(\Ren)$, then   $W(\f_2,\f_1)\in M^{p}_{1\otimes v_{s}}(\Renn)$
	by  \eqref{wigest}.
	Since $a\in M^{p,\infty } _{1\otimes v_{-s}}$, the convolution relation
	$M^{p,\infty} _{1\otimes v_{-s}}(\rdd) \ast M^{p}_{1\otimes v_{s}}(\rdd) \subseteq
	M^{p}(\rdd)$ of Proposition \ref{mconvmp} implies that  the Weyl symbol
	$\sigma=a\ast W(\f_2,\f_1)$ is in $M^{p}(\rdd)$. The  result now
	follows from  Theorem \ref{Charly1} ({\it i}).
	
	The items ({\it ii}) and ({\it ii}) are proved similarly, see \cite[Theorem 3.1]{CG02}.
\end{proof}

\begin{corollary} Any localization operators $\aaf$ with symbol $a$ in $ M^{\infty}_{v_s\otimes 1} (\rdd)$, $s>0$,  and windows $\f_1,\f_2$ in $\cS(\rd)$ is a compact operator belonging to the Schatten class $S_p(\lrd)$, with $p>2d/s$.
\end{corollary}
\begin{proof} It immediately follows from  the inclusion relations for modulation spaces in  \eqref{compact} and the sufficient conditions in Theorem \ref{class}.
\end{proof}

\section{Necessary Conditions}
The necessary conditions for Schatten class localization operators  for the Banach case $p\geq 1$ is contained in the work \cite[Theorem 1 (b)]{CG05}, see also \cite{FG2006} , who recaptured the results in \cite[Theorem 1 (b)]{CG05} by using different techniques.  
Before stating the necessary conditions, observe that using the inclusion relations for modulation spaces in Theorem \ref{inclusionG}, one can rephrase the unweighted  sufficient conditions in Theorem \ref{class} as follows.

\begin{theorem}\label{classLargercond}
	If $1 \leq p  \leq \infty$, then  the mapping $(a,\f _1, \f _2)
	\mapsto \aaf $ is  bounded  from $M^{p,\infty }(\rdd ) \times M^1
	(\rd )\times M^1(\rd )$ into $S_p(\lrd)$, i.e.,
	$$\|\gaw\|_{S_p}\leq C
	\|a\|_{M^{p,\infty}}\|\f_1\|_{M^1}\|\f_2\|_{M^1}\,$$ for a
	suitable constant $C>0$.
\end{theorem}
\begin{proof} The inequality immediately follows from Theorem \ref{class} and the estimate $\|\f_2\|_p\leq \|\f_2\|_1$, for any $p> 1$, by the inclusion relation $M^1(\rd)\subset M^p(\rd)$.
\end{proof}

The vice versa of the sufficient conditions above is shown hereafter.
\begin{theorem} \label{main}
Consider $1\leq p\leq \infty$. If $\gaw\in S_p(\lrd)$ for every pair of  windows $\f_1,\f_2\in\cS(\rd)$ with norm estimate 
	\begin{equation}\label{mt}
	\|\gaw\|_{S_{p}}\leq  C\, \|\f_1\|_{ M^1}\, \|\f_2\|_{ M^1},
	\end{equation}
	where the constant $C>0$  depends only on the symbol  $a$, then $a\in M^{p,\infty}(\rdd)$.
\end{theorem}
In what follows we detail the main steps of the proof, in order to underline the tools employed. The  key role is played by Corollary \ref{adj4}, togetheer with the characterization of the $M^{p,\infty}(\rdd)$-norm of the symbol $a$ via Gabor frames.\par 
\vspace{0.3truecm}
\noindent
\emph{  Sketch of the proof of Theorem \ref{main}.}\\
Consider $0<\a,\b<1$, $\Phi\phas= 2^{-d}e^{-x^2-\o^2}\in\cS(\rdd)$ and the  Gabor frame $(T_{\a k}M_{\b n}\Phi)_{n,k\in\zdd}$. We compute the $M^{p,\infty}(\rdd)$-norm of the symbol  $a$ in $\gaw$   by using the norm characterization in \eqref{idea}
\begin{equation}\label{idea2}
\|a\|_{M^{p,\infty}(\rdd)}\asymp \|\la a ,M_{\b n}T_{\a
	k}\Phi\ra_{n,k\in\zdd}\|_{\ell^{p,\infty}(\bZ^{4d})}.
\end{equation}
Using \eqref{Gauss0}  we can write
\begin{equation}\label{Gauss?}\Phi\phas=2^{-d}e^{-\pi
	(x^2+\o^2)}=V_\f \f\phas\overline{V_\f \f\phas}.\end{equation}
Now, let $k=(k_1,k_2), n=(n_1,n_2)\in\zdd$, by \eqref{Gauss?} and
Formula \eqref{bo}, the \tfs \, of $\Phi$ can be expressed by the
point-wise product of two STFTs:
\begin{eqnarray*}M_{\b n}T_{\a k}\Phi\phas&=&M_{ (\b n_1,\b n_2)}T_{ (\a k_1,\a
		k_2)}(V_\f \f\overline{V_\f \f})\phas\\
	&=&V_{(M_{\b n_1}T_{-\b
			n_2} \f)}(M_{\b n_1} T_{-\b
		n_2} M_{\a k_2}T_{\a
		k_1} \f)\cdot\overline{V_\f (M_{\a k_2}T_{\a
			k_1} \f)}.
\end{eqnarray*}
Using the weak definition of localization operator given in
\eqref{anti-Wickg}, we can write
\begin{equation}\label{???}
\la a, M_{\b n}T_{\a k}\Phi\ra=\la A_a^{\f,(M_{\b n_1}T_{-\b
		n_2} \f)}(M_{\a k_2}T_{\a
	k_1} \f), M_{\b n_1} T_{-\b
	n_2} M_{\a k_2}T_{\a
	k_1} \f\ra.
\end{equation}
The $M^{p,\infty}$-norm of the symbol $a$ can be recast as
\begin{eqnarray*}
	\|a\|_{M^{p,\infty}}\!\!&\asymp&\|\la a,M_{\b n}T_{\a
		k}\Phi\ra_{n,k\in\zdd}\|_{\ell^{p,\infty}(\bZ^{4d})}\\
	&=&\sup_{n\in\zdd}\left(\sum_{k \in\zdd}|\la a,M_{\b n}T_{\a
		k}\Phi\ra|^p\right)^{1/p}\\
	&=&\!\!\!\sup_{(n_1,n_2)\in\zdd}\!\left(\sum_{(k_1,k_2)
		\in\zdd}\!\!|\la A_a^{\f,(M_{\b n_1}T_{-\b
			n_2} \f)}(M_{\a k_2}T_{\a
		k_1} \f), M_{\b n_1} T_{-\b
		n_2} M_{\a k_2}T_{\a
		k_1} \f\ra|^p\right)^{1/p}
\end{eqnarray*}
We apply the assumption \eqref{mt} to the
localization operators $A_a^{\f,(M_{\b n_1}T_{-\b
		n_2} \f)}$; in fact, for every choice of
	$\b,n_1,n_2$, the functions $M_{\b n_1}T_{-\b
n_2} \f$  are in the Schwartz class  $\cS(\rd)$, so that the localization operators satisfy the uniform estimate
\begin{equation}\label{e3} \| A_a^{\f,(M_{\b n_1}T_{-\b
		n_2} \f)}\|_{S_p}\le C \|\f\|_{M^1} \|M_{\b n_1}T_{-\b
	n_2} \f\|_{M^1}= C\|\f\|_{M^1}^2,\end{equation}
since the time-frequency shifts are isometry on $M^1(\rd)$.

Finally, applying  Corollary \ref{adj4} with the Gabor frame  $(M_{\a k_2} T_{\a k_1} \f )_{k_1,k_2\in\zd}$ and operators $T=A_a^{\f,(M_{\b n_1}T_{-\b
		n_2} \f)}\in S_p$ and $L=M_{\b n_1} T_{-\b
	n_2}\in S_\infty$,
we can majorize the norm $\|a\|_{M^{p,\infty}}$ as 
\begin{align*}
	&\|a\|_{M^{p,\infty}}\\
	&\,\,\asymp\!\!\sup_{(n_1,n_2)\in\zdd}\!\!\|\la
	A_a^{\f,(M_{\b n_1}T_{-\b
			n_2} \f)}(M_{\a k_2}T_{\a
		k_1} \f), M_{\b n_1} T_{-\b
		n_2} M_{\a k_2}T_{\a
		k_1} \f\ra_{(k_1,k_2) \in\zdd}\|_{\ell^p(\zdd)}\\
	&\,\,\lesssim \sup_{(n_1,n_2)\in\zdd}\| A_a^{\f,(M_{\b n_1}T_{-\b
			n_2} \f)}\|_{S_p}\\
	&\,\,\lesssim \sup_{(n_1,n_2)\in\zdd} \|\f\|^2_{M^1}= \|\f\|^2_{M^1}<\infty,
\end{align*}
where in the last inequality we used \eqref{e3}.
\endproof\\
\subsection{Conclusion and Perspectives}
As it becomes clear from the previous proof, we cannot expect to prove necessary conditions for small $p$, that is $0<p<1$, using similar techniques to the case $p\geq 1$. The main obstruction being the fact that Corollary \ref{adj4} does not hold for $0<p<1$. Observe that the discrete modulation norm via Gabor frames in \eqref{idea2} remains valid  also for $0<p<1$.
In view of the sufficient conditions in Theorem \ref{class}, we conjecture that a necessary condition of the type expressed below should hold true.
\begin{theorem}\label{classLargercond}
	For $0<p<1$, if  $\gaw$ is in $S_p(\lrd)$ for every pair of windows $\f_1,\f_2\in\cS(\rd)$ and there exists a $C>0$ such that 
	$$\|\gaw\|_{S_p}\leq C
	\|\f_1\|_{M^p}\|\f_2\|_{M^p},\quad \f_1,\f_2\in\cS(\rd),$$ then $a\in M^{p,\infty}(\rdd)$.
\end{theorem}

\section*{Acknowledgments} The author wish to thank Prof. Fabio Nicola  for his suggestions and comments.


\begin{thebibliography}{99}
	
	\bibitem{Abreu2012}
	L.~D. Abreu, M.~D\"orfler:
	\newblock An inverse problem for localization operators.
	\newblock { Inverse Problems}, \textbf{28}(11), 115001, 16 (2012)
	
	
	\bibitem{Abreu2016}
	L.~D. Abreu, K.~Gr\"{o}chenig, and J.~L. Romero.
	\newblock On accumulated spectrograms.
	\newblock {Trans. Amer. Math. Soc.}, \textbf{368}(5):3629--3649, (2016)
	
	\bibitem{Abreu2017}
	L.~D. Abreu, J.~a.~M. Pereira, and J.~L. Romero.
	\newblock Sharp rates of convergence for accumulated spectrograms.
	\newblock { Inverse Problems}, \textbf{33}(11), 115008, 12 (2017)
	
	
	\bibitem{BCN19} F. Bastianoni,  E. Cordero, F. Nicola:
	\newblock Decay and Smoothness for Eigenfunctions of Localization Operators. { Submitted}. ArXiv:1902.03413
	\bibitem{Berezin71}
	F.~A. Berezin:
	\newblock Wick and anti-{W}ick symbols of operators.
	\newblock { Mat. Sb. (N.S.)},  \textbf{86}(128), 578--610 (1971).
	\bibitem{BCG2004} P.~Boggiatto, E.~Cordero,  K.~Gr\"ochenig: Generalized {A}nti-{W}ick operators with symbols in distributional {S}obolev spaces. {Integral Equations and Operator Theory}, \textbf{48}(4), 427--442 (2004)
	
	\bibitem{CG02} E. Cordero, 
	K. Gr\"ochenig: {T}ime-frequency analysis
	of {L}ocalization operators,  J. Funct. Anal.,
	\textbf{205}(1), 107--131 (2003)
	\bibitem{CG05} E. Cordero, K.
	Gr\"ochenig: Necessary Conditions for Schatten Class Localization Operators,  Proc. Amer. Math. Soc.,
	\textbf{133}(12), 3573--3579 (2005)
	\bibitem{Wignersharp2018}
	E.~Cordero, F.~Nicola:
	\newblock Sharp integral bounds for {W}igner distributions.
	\newblock {Int. Math. Res. Not. IMRN}, \textbf{6}, 1779--1807 (2018)
	\bibitem{cordero2} E. Cordero, F. Nicola: Pseudodifferential operators on $L^p$, Wiener amalgam and modulation spaces. { Int. Math. Res. Notices}, \textbf{10}, 1860--1893 (2010)
	
	
	\bibitem{DB1}
	I.~Daubechies:
	\newblock Time-frequency localization operators: a geometric phase space
	approach.
	\newblock { IEEE Trans. Inform. Theory}, \textbf{34}(4), 605--612, (1988)
	
	\bibitem{DB2}
	I.~Daubechies and T.~Paul:
	\newblock Time-frequency localization operators---a geometric phase space
	approach. {II}. {T}he use of dilations.
	\newblock { Inverse Problems}, \textbf{4}(3), 661--680 (1988).
	
	\bibitem{feichtinger-modulation}
	H.~G. Feichtinger:
	\newblock Modulation spaces on locally compact abelian groups.
	\newblock In {\em Technical report, University of Vienna, 1983, and also in
		``Wavelets and Their Applications''}, pages 99--140. M. Krishna, R. Radha, S.
	Thangavelu, editors, Allied Publishers (2003)
	
	\bibitem{FG2006}
	C. Fern\'{a}ndez and A. Galbis.
	\newblock  Compactness of time-frequency localization operators on
	{$L^2(\mathbb{R}^d)$}.
	\newblock {J. Funct. Anal.}, \textbf{233}(2), 335--350 (2006)
	
	
	
	\bibitem{Galperin2014}
	Y. V. Galperin:
	\newblock { Young's Convolution Inequalities for weighted mixed (quasi-) norm spaces}. {J. Ineq. Spec. Funct.}, \textbf{5}(1), 1--12, 2014.
	
	\bibitem{Galperin2004}
	Y.~V. Galperin, S.~Samarah:
	\newblock Time-frequency analysis on modulation spaces {$M^{p,q}_m$}, {$0<p,\
		q\leq\infty$}.
	\newblock {Appl. Comput. Harmon. Anal.}, \textbf{16}(1), 1--18 (2004)
	\bibitem{Gohberg1969}
	I. C. Gohberg  and M. G. Kre\u{i}n:
	\newblock {\em Introduction to the theory of linear nonselfadjoint operators}
	\newblock American Mathematical Society, Providence, R.I. (1969)
	
	\bibitem{Maurice2015}
	M.~A. de~Gosson:
	\newblock The canonical group of transformations of a {W}eyl-{H}eisenberg
	frame; applications to {G}aussian and {H}ermitian frames.
	\newblock { J. Geom. Phys.}, \textbf{114}:375--383 (2017)
	
	\bibitem{deGossonATFA19}
	M. ~A. de Gosson:
	\newblock {Generalized Anti-Wick Quantum States.}
	\newblock Landscapes of Time-frequency Analysis, Birkh\"auser/Springer Basel AG, Basel, to appear
	\bibitem{gro96}
	K.~Gr{\"o}chenig:
	\newblock An uncertainty principle related to the {P}oisson summation formula.
	\newblock { Studia Math.}, \textbf{121}(1), 87--104 (1996)
	\bibitem{grochenig} K. Gr\"ochenig: \textit{Foundation of Time-Frequency Analysis},
	Birkh\"auser, Boston MA (2001)
	\bibitem{GH99}
	K.~Gr{\"o}chenig, C. Heil:
	\newblock  Modulation spaces and pseudodifferential operators, Integral Equations Operator Theory \textbf{34}(4), 439--457 (1999)
	
	
	
	\bibitem{Guo2019}
	W. Guo, J. Chen, D. Fan,  G. Zhao:
	\newblock	Characterizations of Some Properties on Weighted Modulation and Wiener Amalgam Spaces
	\newblock { Michigan Math. J.}, \textbf{68}, 451--482 (2019)
	
	\bibitem{HeWeiss96}
	E. Hern{\'a}ndez, G. Weiss:
	\newblock {\em A first course on wavelets}
	\newblock CRC Press (1996)
	
	
	\bibitem{ZhuSchatten2015}
	B. Hu, L. H.  Khoi, k. Zhu:
	\newblock Frames and operators in {S}chatten classes.
	\newblock Houston J. Math., \textbf{41}(4), 1191--1219 (2015)
	
	
	\bibitem{Reed-Simon1975}
	M. Reed and B. Simon:
	\newblock {\em Methods of modern mathematical physics. {II}. {F}ourier
		analysis, self-adjointness} Academic Press [Harcourt Brace Jovanovich, Publishers], New York-London (1975)
	\bibitem{Simon2005}
	B. Simon:
	\newblock{\em Trace ideals and their applications}, volume 120 of Mathematical Surveys and
	Monographs. American Mathematical Society, Providence, R.I (2005)
	\bibitem{Schatten70}
	R. Schatten:
	\newblock{\em   Norm ideals of completely continuous operators}. Second printing. Ergebnisse der Mathematik und ihrer Grenzgebiete, Band 27. Springer-Verlag, Berlin (1970)
	
	\bibitem{seip92}
	K.~Seip:
	\newblock Density theorems for sampling and interpolation in the
	{B}argmann-{F}ock space. {I}.
	\newblock { J. Reine Angew. Math.}, \textbf{429}, 91--106 (1992)
	
	\bibitem{seip-wallsten}
	K.~Seip, R.~Wallst{\'e}n:
	\newblock Density theorems for sampling and interpolation in the
	{B}argmann-{F}ock space. {I}{I}.
	\newblock { J. Reine Angew. Math.}, \textbf{429}, 107--113 (1992)
	
	
	\bibitem{Shubin91}
	M.~A. Shubin:
	\newblock {\em Pseudodifferential operators and spectral theory}.
	\newblock Springer-Verlag, Berlin, second edition (2001)
	\bibitem{Sjo94}
	J.~Sj{\"o}strand:
	\newblock An algebra of pseudodifferential operators.
	\newblock {Math. Res. Lett.}, \textbf{1}(2), 185--192 (1994)

	
	\bibitem{Nenad2015}
	N.~Teofanov.
	\newblock Gelfand-{S}hilov spaces and localization operators.
	\newblock {\em Funct. Anal. Approx. Comput.}, \textbf{7}(2), 135--158 (2015)
	
	\bibitem{Nenad2016}
	N.~Teofanov.
	\newblock Continuity and {S}chatten--von {N}eumann properties for localization
	operators on modulation spaces.
	\newblock {\em Mediterr. J. Math.}, \textbf{13}(2), 745--758 (2016)
	

	\bibitem{toft2004} J. Toft: 
	\newblock Continuity properties for modulation spaces, with
	applications to pseudo-differential calculus. I, J. Funct. Anal., \textbf{207}(2), 399--429 (2004)
	\bibitem{ToftquasiBanach2017}
	J.~Toft:
	\newblock Continuity and compactness for pseudo-differential operators
	with symbols in quasi-{B}anach spaces or {H}\"{o}rmander classes.
	\newblock { Anal. Appl. (Singap.)}, \textbf{15}(3), 353--38 (2017)
	\bibitem{Wong02}
	M.~W. Wong:
	\newblock { Wavelet transforms and localization operators}, volume 136 of
	{Operator Theory: Advances and Applications}.
	\newblock Birkh\"auser Verlag, Basel (2002)
	\bibitem{zhu2007operator}
	K. Zhu:
	\newblock{Operator Theory in Function Spaces}, {Mathematical surveys and monographs}, American Mathematical Society (2007)
	
\end{thebibliography}
\end{document}